\documentclass[12pt]{amsart}
\usepackage{amsmath}
\usepackage[backref=page]{hyperref}
\usepackage{tikz,float,caption}
\usetikzlibrary{decorations.markings}

\usepackage[margin=1.25 in,top=1.2in, bottom=1.4 in]{geometry}

\usepackage[bb=ams,scr=euler]{mathalpha}

\usepackage{setspace}
\usepackage{enumitem}
\setenumerate{
  wide=10pt,
  leftmargin=0pt,
  labelwidth=*,
  label=(\roman*),
  topsep=0pt,
  listparindent=0pt,
  parsep=4pt}

\usepackage{thmtools}
\declaretheoremstyle[headfont=\normalsize\normalfont\bfseries,notefont=\mdseries, notebraces={(}{)},bodyfont=\normalfont,postheadspace=0.5em]{basicstyle}

\declaretheorem[name=Definition,style=basicstyle,qed=$\triangle$]{definition}
\declaretheorem[name=Remark,style=basicstyle,sibling=definition]{remark}
\declaretheorem[style=basicstyle,name=Theorem]{theorem}
\declaretheorem[style=basicstyle,name=Lemma,sibling=theorem]{lemma}
\declaretheorem[style=basicstyle,name=Proof,numbered=no,qed=$\square$]{preproof}
\renewenvironment{proof}{\preproof}{\endpreproof}

\makeatletter
\newcommand{\abs}[1]{\left|#1\right|}
\newcommand{\bd}{\partial}
\newcommand{\cl}[1]{\overline{#1}}
\newcommand{\C}{\mathbb{C}}
\renewcommand{\d}{\mathrm{d}}
\newcommand{\ip}[1]{\left\langle#1\right\rangle}
\newcommand{\norm}[1]{\left\lVert#1\right\rVert}
\newcommand{\N}{\mathbb{N}}
\newcommand{\R}{\mathbb{R}}
\renewcommand\section{\@startsection{section}{1}{0pt}{-3.5ex \@plus -1ex \@minus -.2ex}{2.3ex \@plus.2ex}{\centering\bfseries}}
\newcommand{\set}[1]{\left\{#1\right\}}
\newcommand{\Z}{\mathbb{Z}}

\makeatother

\makeatother
\begin{document}

\title[Holomorphic Cylinders, Exponential Estimates, and Flow Lines]{Perturbed Holomorphic Cylinders, Exponential Estimates, and Flow Lines}


\author{Dylan Cant}
\address{Stanford University}
\email{dcant@stanford.edu}
\thanks{The author would like to thank Daren Chen for many useful discussions surrounding this result.}



\date{Friday, October 8, 2021}

\dedicatory{Dedicated to Ollie}


\begin{abstract}
    A convergence result for perturbed holomorphic cylinders in $\R^{2n}$ is proved; a sequence of perturbed holomorphic cylinders converges to a configuration of two holomorphic disks joined by a flow line. A key step in the computation is played by an a priori exponential estimate. 
\end{abstract}

\maketitle

\section{Introduction}

\begin{definition}[geometric set-up]\label{definition:set-up}
  Let $V_{n}$ be a $C^{\infty}$ convergent sequence of vector fields on $\cl{B}(1)^{2n}\subset \R^{2n}$, $\epsilon_{n}\to 0$, and $u_{n}:[-r_{n}-1,r_{n}+1]\times \R/\Z\to \cl{B}(1)$ satisfy the following partial differential equation:
  \begin{equation}\label{eq:1}
    \bd_{s}u_{n}+J_{0}\bd_{t}u_{n}=\epsilon_{n}V_{n},
  \end{equation}
  where $J_{0}$ is the standard complex structure on $\R^{2n}\simeq \C\times \C\times \cdots \times \C$.

\begin{figure}[H]
  \centering
  \begin{tikzpicture}
    \begin{scope}[shift={(-2,0)}]
      \draw (0,0) circle (0.5 and 0.2) (0.5,2) circle (0.5 and 0.2);
      \draw (-0.5,0) to[out=90,in=-90] (0,2);
      \draw (0.5,0) to[out=90,in=-90] (1,2);
    \end{scope}
    \begin{scope}
      \draw (0,0) circle (0.5 and 0.2) (0.5,2) circle (0.5 and 0.2);
      \draw (-0.5,0) to[out=90,in=-90] (0.25,1) to[out=90,in=-90] (0,2);
      \draw (0.5,0) to[out=90,in=-90] (0.35,1) to[out=90,in=-90] (1,2);
    \end{scope}
    \begin{scope}[shift={(2,0)}]
      \draw (0,0) circle (0.5 and 0.2) (0.5,2) circle (0.5 and 0.2);
      \draw (-0.5,0) to[out=90,in=90]node(A)[draw,circle,inner sep=1pt,fill=black]{} (0.5,0) (0,2) to[out=-90,in=-90]node(B)[draw,circle,inner sep=1pt,fill=black]{} (1,2);
      \draw[postaction={decorate,decoration={markings,mark=at position 0.5 with {\arrow[scale=1.5]{>};}}}] (A)to[out=90,in=-90]node[right]{flow line for $V$}(B);
      
    \end{scope}
  \end{tikzpicture}
  \caption{Degeneration of perturbed holomorphic cylinders. Here $V=\lim_{n}V_{n}$.}
  \label{fig:1}
\end{figure}
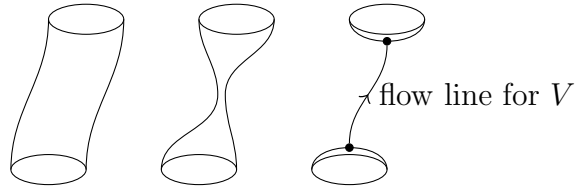
  We think of solutions to \eqref{eq:1} as being perturbed holomorphic cylinders. If $V_{n}$ is a gradient vector field, then this equation is \emph{Floer's equation} (restricted to a finite length cylinder) (see \cite{floer-ham} and \cite{salamon1997}).
\end{definition}

The main goal in this short paper is to explain how a sequence of maps $u_{n}$ satisfying Definition \ref{definition:set-up} has a subsequence that converges (in a certain sense) to a flow line for $V=\lim_{n\to\infty}V_{n}$ joining two holomorphic disks (see Figure \ref{fig:1}). More precisely, we will show that if $\epsilon_{n}r_{n}$ converges to a finite number $\ell>0$, or $r_{n}\to \infty$ and $\epsilon_{n}r_{n}\to \ell=0$, then (a subsequence of) $u_{n}$ will converge to two holomorphic disks joined by a flow line for $V$ of length $2\ell$.

Here is the theorem we will prove.
\begin{theorem}\label{theorem:1}
  Let $u_{n}$ be a sequence satisfying Definition \ref{definition:set-up}. Suppose $\epsilon_{n}r_{n}\to \ell$, and if $\ell=0$ then  $r_{n}\to\infty$. Then there exist numbers $\rho_{n}\to \infty$ so that $\epsilon_{n}\rho_{n}\to 0$, $\rho_{n}<r_{n}$, and so that if we decompose the domain:
  \begin{equation*}
    [-r_{n},r_{n}]\times \R/\Z=([-r_{n},-r_{n}+\rho_{n}]\cup [-(r_{n}-\rho_{n}),r_{n}-\rho_{n}]\cup  [r_{n}-\rho_{n},r_{n}])\times \R/\Z,
  \end{equation*}
  and pass to a subsequence, we have:
  \begin{enumerate}
  \item the maps $u_{n}^{-}(s,t)=u_{n}(s-r_{n},t)$ and $u_{n}^{+}(s,t)=u_{n}(s+r_{n},t)$, when restricted to $[0,\rho_{n}]\times \R/\Z$ and $[-\rho_{n},0]\times \R/\Z$, converge uniformly on to holomorphic cylinders defined on $[0,\infty)\times \R/\Z$ and $(-\infty,0]\times \R/\Z$ with removable singularities $x_{-}$ and $x_{+}$.
  \item the rescaled maps
    \begin{equation*}
      v_{n}(s,t)=u_{n}(\epsilon_{n}^{-1}s,\epsilon_{n}^{-1}t)
    \end{equation*}
    defined on the rescaled central cylinder $[-\epsilon_{n}(r_{n}-\rho_{n}),\epsilon_{n}(r_{n}-\rho_{n})]\times \R/\epsilon_{n}\Z$ converges uniformly to a flow line $v_{\infty}$ of $V=\lim_{n}V_{n}$ of length $2\ell$ joining $x_{-}$ and $x_{+}$.
  \end{enumerate}
  See Figure \ref{fig:limiting-configuration}.
\end{theorem}
\begin{figure}[h]
  \centering
\begin{tikzpicture}
  \draw (0,0) circle (0.3 and 1) (5,0) circle (0.3 and 1);
  \draw (0,1) arc (90:-90:1) node[below] {$u^{-}_{\infty}$}(5,1) arc (90:270:1)node[below]{$u^{+}_{\infty}$};
  \path (1,0)node[above right]{$x_{-}$}--(4,0)node[below left]{$x_{+}$};
  \draw[postaction={decorate,decoration={markings,mark=at position 0.5 with{\arrow[scale=1.5]{>}};}}] (1,0) to[out=-30,in=140]node[pos=0.4,below]{$v_{\infty}$} (4,0);
  \path[every node/.style={draw,circle,inner sep=1pt,fill}] (1,0)node{}--(4,0)node{};
\end{tikzpicture}
  \caption{Converging to the limiting configuration.}
  \label{fig:limiting-configuration}
\end{figure}
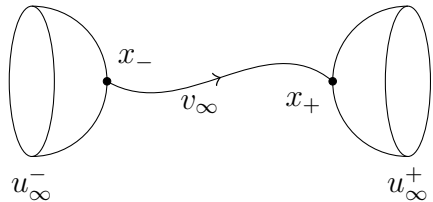
\begin{proof}
  Theorem \ref{theorem:1} is a combination of Lemmas \ref{lemma:endplate-convergence}, \ref{lemma:cutting}, and \ref{lemma:final-conv}, proved below.
\end{proof}

\begin{remark}
  Other researchers have considered this problem previously. See \cite{O3} for a similar compactness result and see \cite{OZ1} and \cite{OZ2} for an analysis of the reverse \emph{gluing} process. The paper \cite{OZ2} also shows exponential convergence to a gradient flow segment. See also \cite{MT} for a related compactness result which also exhibits the formation of gradient flow lines. These papers assume the energy of the holomorphic cylinders is small, while in this paper we assume that the image of $u_{n}$ lies inside of a ball $\cl{B}(1)$ (i.e.,\ we assume a $C^{0}$ bound whereas they assume an $L^{2}$ bound on the derivative).

  A similar compactness phenomenon in the context of harmonic mappings is proved in \cite{ct}, where geodesic segments appear rather than gradient flow lines.
\end{remark}
\begin{remark}
  In the course of our argument, we will need to pass to sub-cylinders a few times to apply the a priori estimates from Section \ref{sec:apriori}. As a consequence, our description of the limit is only valid on $[-r_{n},r_{n}]\times \R/\Z$. In applications, some other information will be needed to control the behavior of $u_{n}$ on the ends (e.g.,\ $u_{n}$ is the restriction of a map defined on a larger Riemann surface $\Sigma_{n}$ equipped with an embedding $[-r_{n}-1,r_{n}+1]\times \R/\Z\to \Sigma_{n}$, or $u_{n}$ has totally real boundary conditions on the two boundary circles).
\end{remark}
\begin{remark}
  If $\epsilon_{n}r_{n}$ converges to $\infty$, then we expect any convergence statement to require topological restrictions on $V$ (e.g.,\ if $V$ is a gradient vector field for a Morse function, then $u_{n}$ will converge to a broken flow line; or, if $V_{n}$ is a Morsification of a Morse-bott vector field $V$, then $u_{n}$ will converge to a Morse-Bott broken flow line, etc).
\end{remark}
\begin{remark}
  The set-up in Definition \ref{definition:set-up} is a bit of a ``toy problem'' since we do not consider more general almost complex structures $J$. The author expects more general almost complex structures can be handled in a similar way. The estimates/computations will involve many additional junk terms involving the derivatives of $J$ (i.e.,\ some things which cancel exactly in this paper will not cancel exactly with a more general $J$). Hopefully the ideas conveyed in this paper will still be of use to readers interested in more general almost complex structures.
\end{remark}

\section{A Priori Estimates}
\label{sec:apriori}
In this section we will prove various a priori estimates for maps $u_{n}$ satisfying Definition \ref{definition:set-up}. Our first result is a \emph{uniform a priori estimate} for $u_{n}$ and its derivatives. 
The elliptic estimates for the ``delbar'' operator $\cl{\bd}=\bd_{s}+J_{0}\bd_{t}$ imply that for every $k\in \mathbb{N}$ and $\delta>0$ there is a constant $C^{\mathrm{ell}}_{\delta,k}$ so that for any smooth $\R^{2n}$-valued function $u$ we have:
\begin{equation}\label{eq:elliptic-estimate}
  \norm{u}_{W^{k+1,2}(s+\Omega(\delta))}\le C^{\mathrm{ell}}_{\delta,k}(\norm{(\bd_{s}+J_{0}\bd_{t})u}_{W^{k,2}(s+\Omega(2\delta))}+\norm{u}_{W^{k,2}(s+\Omega(2\delta))}),
\end{equation}
where $\Omega(\delta)=[-\delta,\delta]\times \R/\Z$. See \cite[Lemma C.1]{robbinsalamon} for a short proof. In particular, the fact that $V_{n}$ converges in $C^{\infty}$ and $\abs{u_{n}}$ is everywhere less than $1$, together with the Sobolev embedding theorem, implies the following following result:
\begin{lemma}\label{lemma:unif-est}
  Assuming the set-up in Definition \ref{definition:set-up}, we have the following a priori estimates for the behavior of $u_{n}$ on the sub-cylinder $[-r_{n},r_{n}]\times \R/\Z$:
  \begin{equation}\label{eq:apriori3}
    \limsup_{n\to\infty}\sup_{(s,t)\in [-r_{n},r_{n}]\times \R/\Z}\abs{\nabla^{k}u_{n}(s,t)}\le C_{k},
  \end{equation}
  where $C_{k}$ depends only on $k$.
\end{lemma}
\begin{proof}
  It is straightforward to show that
  \begin{equation*}
    \limsup_{n\to\infty}\norm{u_{n}}_{W^{k,2}(s+\Omega(2^{-k}))}\le C_{k}^{\prime}.
  \end{equation*}  
  Then the Sobolev embedding theorem (see \cite[Theorem B.1.11]{mcduffsalamon}) gives $$\abs{\nabla^{k}u(s,t)}\le \norm{u_{n}}_{C^{k}(s+\Omega(2^{-k-2}))}\le C^{\mathrm{Sob}}_{k}\norm{u_{n}}_{W^{k+2,2}(s+\Omega(2^{-k-2}))}.$$
  These two observations complete the proof.
\end{proof}
Next, we will prove an \emph{exponential type estimate} that shows $u_{n}$ becomes very thin as we approach the center of the cylinder. 

\begin{definition}\label{definition:center-of-mass}
  Introduce the \emph{center of mass} $q_{n}:[-r_{n}-1,r_{n}+1]\to B(1)$ given by 
  \begin{equation*}
    q_{n}(s)=\int_{\R/\Z}u_{n}(s,t)\,\d t.
  \end{equation*}
  Assuming that $u_{n}$ satisfies Definition \ref{definition:set-up}, then we compute
  \begin{equation}\label{eq:perturb-flow}
    q_{n}^{\prime}(s)-\epsilon_{n}V_{n}(q_{n}(s))=\epsilon_{n}\int_{\R/\Z}V_{n}(u_{n}(s,t))-V_{n}(q_{n}(s))\,\d t.
  \end{equation}
  We think of of \eqref{eq:perturb-flow} as a perturbation of the flow line equation for $V_{n}$.
\end{definition}

Introduce the function
\begin{equation}\label{eq:definition-gamma}
  \gamma_{n}(s)=\frac{1}{2}\int_{\R/\Z}\abs{u_{n}(s,t)-q_{n}(s)}^{2}\,\d t.
\end{equation}
Our goal is to show that $\gamma_{n}^{\prime\prime}(s)-\delta^{2}\gamma_{n}(s)\ge 0$ for some $\delta>0$. As explained in the pre-print \cite[Section 4.4]{cant2021exponential} the estimate $\gamma^{\prime\prime}-\delta^{2}\gamma\ge 0$ is related to exponential decay estimates.

\begin{remark}
  For simplicity of notation, we introduce the norm:
  \begin{equation*}
    \norm{f}^{2}=\int_{\R/\Z}\abs{f(t)}^{2}\,\d t.
  \end{equation*}
\end{remark}
\begin{lemma}\label{lemma:diff-ineq}
  There is a constant $\delta>0$, independent of $n$ and $u_{n}$, so that $\gamma_{n}$ satisfies the following differential inequality
  \begin{equation}\label{eq:diff-ineq}
    \gamma_{n}''(s)-\delta^{2}\gamma_{n}(s)\ge \frac{3}{4}\norm{\bd_{s}(u_{n}-q_{n})}^{2}+\frac{1}{4}\norm{\bd_{t}(u_{n}-q_{n})}^{2}
  \end{equation}
   for $s\in [-r_{n},r_{n}]$ and for $n$ sufficiently large. More precisely, there is $N$ depending on $\set{\epsilon_{n}}$ so that \eqref{eq:diff-ineq} holds for $n\ge N$ and $s\in [-r_{n},r_{n}]$.
 \end{lemma}

 \begin{proof}
  We compute
  \begin{equation*}
    \gamma_{n}^{\prime\prime}(s)=\norm{\bd_{s}(u_{n}-q_{n})}^{2}+\int_{\R/\Z}\ip{u_{n}-q_{n},\bd_{s}\bd_{s}(u_{n}-q_{n})}\d t.
  \end{equation*}
  We have
  \begin{equation}\label{eq:holo-avg}
    \bd_{s}(u_{n}-q_{n})=-J_{0}\bd_{t}(u_{n}-q_{n})+\epsilon_{n}(V_{n}(u_{n})-\int_{\R/\Z} V_{n}(u_{n})\,\d t),
  \end{equation}
  hence
  \begin{equation*}
    \bd_{s}\bd_{s}(u_{n}-q_{n})=-\bd_{t}\bd_{t}(u_{n}-q_{n})+\epsilon_{n}(\bd_{s}-J_{0}\bd_{t})[V_{n}(u_{n})-\int_{\R/\Z}V_{n}(u_{n})\,\d t].
  \end{equation*}
  Integration by parts then yields
  \begin{equation}\label{eq:integration-by-parts}
    \gamma_{n}^{\prime\prime}(s)=\norm{\bd_{s}(u_{n}-q_{n})}^{2}+\norm{\bd_{t}(u_{n}-q_{n})}^{2}+\epsilon_{n}\int\ip{u_{n}-q_{n},\beta_{n}}\d t,
  \end{equation}
  where
  \begin{equation*}
    \beta_{n}=(\bd_{s}-J_{0}\bd_{t})[V_{n}(u_{n})-\int_{\R/\Z}V_{n}(u_{n})\,\d t].
  \end{equation*}
  We rewrite $\beta_{n}$ as follows:
  \begin{equation*}
    \begin{aligned}
      V_{n}(u_{n})-\int_{\R/\Z}V_{n}(u_{n})&=V_{n}(u_{n})-V_{n}(q_{n})+\int_{\R/\Z}V_{n}(u_{n})-V_{n}(q_{n})\,\d t\\
      &=A_{n}\cdot (u_{n}-q_{n}),
    \end{aligned}
  \end{equation*}
  where
  \begin{equation}\label{eq:An-derivate}
    A_{n}=\int_{0}^{1}\d V_{n}((1-\tau)q_{n}+\tau u_{n})\d \tau-\int_{\R/\Z}\int_{0}^{1}\d V_{n}((1-\tau)q_{n}+\tau u_{n})\d \tau\d t.
  \end{equation}
  As a consequence of the a priori estimates \eqref{eq:apriori3} we conclude that $A_{n}$ is uniformly bounded in $C^{1}$ on the cylinder $[-r_{n},r_{n}]\times \R/\Z$. For $n$ sufficiently large (say $n\ge N$), we can bound its derivatives by a constant independent of $u_{n}$.

  Since
  \begin{equation*}
    \beta_{n}=((\bd_{s}-J_{0}\bd_{t})A_{n})\cdot (u_{n}-q_{n})+A_{n}\cdot ((\bd_{s}-J_{0}\bd_{t})(u_{n}-q_{n})),
  \end{equation*}
  we conclude that for $s\in [-r_{n},r_{n}]$
  \begin{equation*}
    \abs{\beta_{n}}\le C_{1}(\abs{u_{n}-q_{n}}+\abs{\bd_{s}(u_{n}-q_{n})}+\abs{\bd_{t}(u_{n}-q_{n})}),
  \end{equation*}
  where $C_{1}$ is independent of $n\ge N$ and $u_{n}$. By applying the Cauchy-Schwarz inequality, and the inequality $2ab \le a^{2}+b^{2}$, we conclude from \eqref{eq:integration-by-parts} that
  \begin{equation*}
    \begin{aligned}
      \gamma_{n}^{\prime\prime}(s)&\ge \norm{\bd_{s}(u_{n}-q_{n})}^{2}+\norm{\bd_{t}(u_{n}-q_{n})}^{2}\\&\hspace{1cm}-\epsilon_{n}C_{2}(\norm{u_{n}-q_{n}}^{2}+\norm{\bd_{s}(u_{n}-q_{n})}^{2}+\norm{\bd_{t}(u_{n}-q_{n})}^{2}),
    \end{aligned}
  \end{equation*}
  where $C_{2}$ is independent of $n\ge N$ and $u_{n}$. The next step in the argument is to invoke the \emph{Poincar\'e Inequality}, which states that there is a constant $c_{\mathrm{pc}}>0$ so that
  \begin{equation*}
    \text{$f:\R/\Z\to \R^{N}$ has mean zero}\implies\norm{f}^{2}\le c_{\mathrm{pc}}\norm{\bd_{t}f}^{2}.
  \end{equation*}
  A simple proof is available using Fourier analysis. Noting that $u_{n}-q_{n}$ has mean zero (by construction), we conclude that we can pick $n$ sufficiently large so that $\epsilon_{n}C_{2}$ is less than $\min\set{0.25 c_{\mathrm{pc}}^{-1},0.25}$, which yields
  \begin{equation*}
    \gamma_{n}''(s)-\frac{1}{4c_{\mathrm{pc}}}\gamma_{n}(s)\ge \frac{3}{4}\norm{\bd_{s}(u_{n}-q_{n})}^{2}+\frac{1}{4}\norm{\bd_{t}(u_{n}-q_{n})}^{2},
  \end{equation*}
  provided $s\in [-r_{n},r_{n}]$. Setting $\delta^{2}=(4c_{\mathrm{pc}})^{-1}$ yields the desired result.
\end{proof}

As in \cite{cant2021exponential}, we use the differential inequality \eqref{eq:diff-ineq} to conclude the following \emph{exponential a priori estimate} on the $W^{1,2}$ size of $u_{n}-q_{n}$:
\begin{lemma}\label{lemma:exponential-estimate}
  Suppose that $u_{n}$ satisfies Definition \ref{definition:set-up} and $q_{n}$ is the center of mass given by Definition \ref{definition:center-of-mass}. There is a constant $C$, independent of $u_{n}$ or $r_{n}$, so that
  \begin{equation*}
    \begin{aligned}      \int_{s-0.5}^{s+0.5}&\norm{u_{n}-q_{n}}^{2}+\norm{\bd_{s}(u_{n}-q_{n})}^{2}+\norm{\bd_{t}(u_{n}-q_{n})}^{2}\d s\\&\hspace{5cm}\le C(e^{-\delta(r_{n}+s)}+e^{-\delta(r_{n}-s)})
    \end{aligned}
  \end{equation*}
  for $s\in [-r_{n}+1,r_{n}-1]$ and for $n$ sufficiently large. Notice that the left hand side is the $W^{1,2}$ norm (squared) over the sub-cylinder $[s-0.5,s+0.5]\times\R/\Z$. Here $\delta>0$ is the constant appearing in the differential inequality \eqref{eq:diff-ineq}.
\end{lemma}
\begin{proof}
  Let $\gamma_{n}$ be the quantity introduced in \eqref{eq:definition-gamma}. Without loss of generality, suppose that $n$ is sufficiently large so that the conclusion of Lemma \ref{lemma:diff-ineq} holds, say $n\ge N$.

  It is clear that $\gamma_{n}(-r_{n})$ and $\gamma_{n}(r_{n})$ are both bounded by some constant $c>0$ (thanks to the estimates \eqref{eq:apriori3}). Then  $\beta_{n}=\gamma_{n}-c(e^{-\delta(r_{n}+s)}+e^{-\delta(r_{n}-s)})$ satisfies $\beta_{n}''-\delta^{2}\beta_{n}\ge 0$, and $\beta_{n}(-r_{n})$ and $\beta_{n}(r_{n})$ are both non-positive.

  As a consequence, $\beta_{n}$ must be non-positive everywhere, otherwise we would conclude a positive maximum, which is precluded by $\beta_{n}''\ge \delta^{2}\beta_{n}$. In particular, we conclude that
  \begin{equation}\label{eq:pointwise-on-gamma}
    \gamma_{n}(s)\le c(e^{-\delta(r_{n}+s)}+e^{-\delta(r_{n}-s)}),
  \end{equation}
  for all $s\in [r_{n}-1,r_{n}+1]$ and for $n\ge N$. Integrating this estimate over $[s-0.5,s+0.5]$, we find a constant $C$ so that
  \begin{equation}\label{eq:one-third}
    \int_{s-0.5}^{s+0.5}\norm{u_{n}-q_{n}}^{2}\d s\le \frac{C}{3}(e^{-\delta(r_{n}+s)}+e^{-\delta(r_{n}-s)}),
  \end{equation}
  for $n\ge N$. It remains to estimate the integrals of $\norm{\bd_{s}(u_{n}-q_{n})}^{2}$ and $\norm{\bd_{t}(u_{n}-q_{n})}$. To do so, we will use a convolution trick. Fix a bump function $\rho:\R\to [0,1]$ supported in $[-1,1]$ and equal to $1$ on $[-0.5,0.5]$. As explained in the caption to Figure \ref{fig:bump-fxn}, this can be done with $\norm{\rho''}_{L^{1}}\le 40$ and $\norm{\rho}_{L^{1}}\le 2$.

\begin{figure}[H]
  \centering
  \begin{tikzpicture}[xscale=1.5,yscale=1.5]
  \draw[opacity=0.3] (-2.99,-0.5) grid[xstep=.5,ystep=.5] (2.99,1.5);
  \draw[opacity=0.3] (-3,-0.5) rectangle (3,1.5);
  \draw[line width=1pt] (-3,0)--(-1,0)to[out=0,in=180](-0.5,1)to[out=0,in=180](0.5,1)to[out=0,in=180](1,0)--(3,0);
  \node at (0.5,-0.5)[below]{$0.5$};
  \node at (-0.5,-0.5)[below]{$-0.5\hphantom{-}$};

  \node at (1,-0.5)[below]{$1$};
  \node at (-1,-0.5)[below]{$-1\hphantom{-}$};
\end{tikzpicture}
  \caption{The bump function $\rho$. We have $\norm{\rho}_{L^{1}}\le 2$. Moreover, this can be achieved with $\norm{\rho''}_{L^{1}}\le 4\pi^{2}+\epsilon\le 40$, because one can take a smooth $\epsilon$-approximation of the $C^{2}$ function $0.5-0.5\cos(2\pi x)$, which interpolates from $0$ to $1$ over an interval of size $0.5$ (we need two copies of this function, one of which is reflected through the $y$-axis).}
  \label{fig:bump-fxn}
\end{figure}
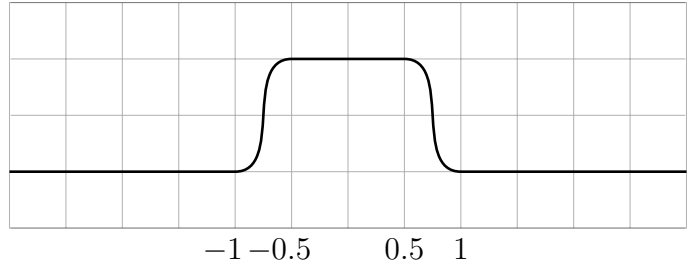
  
  We then convolve both sides of the differential inequality \eqref{eq:diff-ineq} to obtain
  \begin{equation}\label{eq:post-convolution}
    (\rho^{\prime\prime}\ast \gamma_{n})(s)-\delta^{2}(\rho\ast \gamma)(s)\ge \frac{3}{4}\rho\ast \norm{\bd_{s}(u_{n}-q_{n})}^{2}+\frac{1}{4}\rho\ast \norm{\bd_{t}(u_{n}-q_{n})}^{2}.
  \end{equation}

  For functions $f$ supported in $[-1,1]$, it is simple to apply \eqref{eq:pointwise-on-gamma} to estimate
  \begin{equation*}
    \abs{f\ast \gamma}(s)\le \max_{s^{\prime}\in [s-1,s+1]}\abs{\gamma(s^{\prime})}\norm{f}_{L^{1}}\le c\norm{f}_{L^{1}}(e^{-\delta(r+s)}+e^{-\delta(r-s)}).
  \end{equation*}
  This observation, with $f=\rho$ and $f=\rho^{\prime\prime}$ applied to \eqref{eq:post-convolution}, yields
  \begin{equation*}
    \rho\ast (\norm{\bd_{s}(u_{n}-q_{n})}^{2}+\norm{\bd_{t}(u_{n}-q_{n})}^{2})\le 4(40+2\delta^{2})c(e^{-\delta(r+s)}+e^{-\delta(r-s)}).
  \end{equation*}
  Since $\rho=1$ on $[-0.5,0.5]$, we obtain
  \begin{equation*}
    \int_{s-0.5}^{s+0.5}\norm{\bd_{s}(u_{n}-q_{n})}^{2}+\norm{\bd_{t}(u_{n}-q_{n})}^{2}\,\d s\le  \rho\ast (\norm{\bd_{s}(u_{n}-q_{n})}^{2}+\norm{\bd_{t}(u_{n}-q_{n})}^{2}).
  \end{equation*}
  and so, if we pick $C\ge 6(40+2\delta^{2})c$, then we conclude that
  \begin{equation*}
    \int_{s-0.5}^{s+0.5}\norm{\bd_{s}(u_{n}-q_{n})}^{2}+\norm{\bd_{t}(u_{n}-q_{n})}^{2}\,\d s\le\frac{2C}{3}(e^{-\delta(r+s)}+e^{-\delta(r-s)}).
  \end{equation*}
  Combining this with \eqref{eq:one-third} completes the proof of the lemma.  
\end{proof}

We think of Lemma \ref{lemma:exponential-estimate} as being an a priori exponential estimate for the $W^{1,2}$ norm. We will now use the elliptic estimates \eqref{eq:elliptic-estimate} to obtain similar estimates in the $C^{k}$ norm:
\begin{lemma}\label{lemma:apriori-exp-fin}
  Suppose that $u_{n}$ satisfies Definition \ref{definition:set-up} and $q_{n}$ is the center of mass given by Definition \ref{definition:center-of-mass}. For $n$ sufficiently large and each $k\ge 0$, there is a constant $M_{k}$ independent of $u_{n}$ or $r_{n}$ so that
  \begin{equation*}
    \abs{\nabla^{k}(u_{n}-q_{n})(s,t)}\le M_{k}(e^{-c(r_{n}+s)}+e^{-c(r_{n}-s)}),
  \end{equation*}
  for $s\in [-r_{n}+1,r_{n}-1]$, where $c=\delta/2>0$ is the constant appearing in the differential inequality \eqref{eq:diff-ineq}.
\end{lemma}
\begin{proof}
  Let $\Omega_{k}=[2^{-k},2^{k}]\times \R/\Z$. It suffices to show that there are constants $S_{k}$ independent of $u_{n}$ and $r_{n}$ so that
  \begin{equation}\label{eq:sobolev-exponential}
    \norm{u_{n}-q_{n}}_{W^{k,2}(s+\Omega_{k})}\le S_{k}(e^{-c(r_{n}+s)}+e^{-c(r_{n}-s)}),
  \end{equation}
  for $n$ sufficiently large and $s\in [-r_{n}+1,r_{n}-1]$. As in the proof of Lemma \ref{lemma:unif-est}, the Sobolev embedding theorem will then imply the desired $C^{k}$ estimates. Pick $n\ge N$ for $N$ large enough so that the $k=1$ case of \eqref{eq:sobolev-exponential} holds (Lemma \ref{lemma:exponential-estimate}). We will not need to increase $N$ any further.

  Recall from \eqref{eq:holo-avg} that $u_{n}-q_{n}$ satisfies the following equation:
  \begin{equation*}
    \bd_{s}(u_{n}-q_{n})+J_{0}\bd_{t}(u_{n}-q_{n})=\epsilon_{n} A_{n}\cdot(u_{n}-q_{n}),
  \end{equation*}
  where $A_{n}$ is given by \eqref{eq:An-derivate}. Note that the derivatives of $A_{n}$ on $[-r_{n},r_{n}]\times \R/\Z$ are uniformly bounded because of Lemma \ref{lemma:unif-est}.
  
  The elliptic estimates \eqref{eq:elliptic-estimate} can be applied to conclude that\footnote{Here we assume the norms have been defined so that
  \begin{equation*}
    \norm{A_{n}\cdot (u_{n}-q_{n})}_{W^{k,2}}\le \norm{A_{n}}_{C^{k}}\norm{u_{n}-q_{n}}_{W^{k,2}}.
  \end{equation*}
  For some choices of norm, this may hold only up to a combinatorial constant depending on $k$, but this will not affect the argument.
  }
  \begin{equation*}
    \norm{u_{n}-q_{n}}_{W^{k+1,2}(s+\Omega_{k+1})}\le C_{k}^{\mathrm{ell}}(\norm{A_{n}}_{C^{k}(s+\Omega_{k})}+1)\norm{u_{n}-q_{n}}_{W^{k,2}(s+\Omega_{k})}.
  \end{equation*}
  Thus we may recursively define $S_{k}$ by the formula
  \begin{equation*}
    S_{k+1}=C^{\mathrm{ell}}_{k}(B_{k}+1)S_{k},
  \end{equation*}
  where $B_{k}$ is any constant controlling the $C^{k}$ size of $A_{n}$. This establishes \eqref{eq:sobolev-exponential} and therefore completes the proof.  
\end{proof}

\section{Convergence to the limiting configuration}
\label{sec:limiting-configuration}

Our goal in this section is to prove that the sequence $u_{n}$ converges near the ends of the cylinder to holomorphic disks with removable singularities $x_{-}$ and $x_{+}$, and, away from the ends, $u_{n}$ converges to a flow line for $V=\lim_{n\to\infty}V_{n}$ connecting $x_{-}$ to $x_{+}$. See Figure \ref{fig:1} for an illustration of the idea.

As we will see below, in order to capture the limiting flow line, we will need to rescale the map $u_{n}$, i.e.,\ we instead consider the limiting behavior of
\begin{equation*}
  v_{n}(s,t)=u_{n}(\epsilon_{n}^{-1}s,\epsilon_{n}^{-1}t)
\end{equation*}
defined on the rescaled cylinder $[-\epsilon_{n}r_{n},\epsilon_{n}r_{n}]\times \R/\epsilon_{n}\Z$.

Since the domain of $v_{n}$ becomes highly degenerate from a metric perspective, we cannot apply most a priori estimates to $v_{n}$ (e.g.,\ one complication is that disks with a fixed radii $r$ become multiply covered with arbitrarily large covering degree in the domain of $v_{n}$). In order to show that $v_{n}$ converges, we will apply the Arzel\`a-Ascoli theorem.\footnote{In our proof, we will actually apply the Arzel\`a-Ascoli theorem to the center of mass of $v_{n}$. In our case, where $J=J_{0}$, there is a cancellation when deriving the equation for the center of mass (namely $\int_{\R/\Z} J_{0}\bd_{t}u\,\d t=0$) which simplifies the argument.} Unfortunately, the first derivative of $v_{n}$ grows as $\epsilon_{n}^{-1}$ times the first derivative of $u_{n}$. This makes it difficult to bound the derivatives of $v_{n}$ (which is a pre-requisite for applying Arzel\`a-Ascoli).

However, the exponential estimates derived in Section \ref{sec:apriori} \emph{can} be used to bound the derivative of $v_{n}$ on any subcylinder sufficiently far from the boundary circles; we need only travel far enough that $e^{-c(r-s)}+e^{-c(r+s)}\le \epsilon_{n}$.

We begin by analyzing the behavior of $u_{n}$ near the ends of the cylinder. We have the following result:

\begin{lemma}\label{lemma:endplate-convergence}
  Suppose that $u_{n}$ satisfies Definition \ref{definition:set-up}, and $r_{n}\to+\infty$. Consider the translations\footnote{The maximal domain of $u^{-}_{n}$ is $[-1,2r_{n}+1]$ and the maximal domain of $u^{+}_{n}$ is $[-2r_{n}-1,+1]$.}
  \begin{equation*}
    u^{-}_{n}(s,t)=u_{n}(s-r_{n},t)\text{ and }u^{+}_{n}(s,t)=u_{n}(s+r_{n},t).
  \end{equation*}
  After passing to a subsequence,  $u^{-}_{n}$ converges on compact subsets of $[0,\infty)\times \R/\Z$ to a limiting holomorphic\footnote{By ``holomorphic,'' we mean that $\bd_{s}u+J_{0}\bd_{t}u=0$.} map $u^{-}_{\infty}$, and (similarly) $u^{+}_{n}$ converges on compact subsets of $(-\infty,0]\times \R/\Z$ to a limiting holomorphic map $u^{+}_{\infty}$.

  Moreover, $u^{-}_{\infty}(s,t)$ is asymptotic as $s\to\infty$ to a point $x_{-}\in \cl{B}(1)$ and $u^{+}_{\infty}(s,t)$ is asymptotic as $s\to\infty$ to a point $x_{+}\in \cl{B}(1)$ (both limits are uniform in $t$). 
\end{lemma}
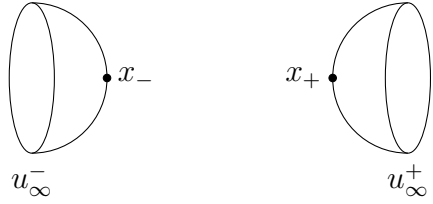
\begin{figure}[h]
  \centering
  \begin{tikzpicture}
  \draw (0,0) circle (0.3 and 1) (5,0) circle (0.3 and 1);
  \draw (0,1) arc (90:-90:1) node[below] {$u^{-}_{\infty}$}(5,1) arc (90:270:1)node[below]{$u^{+}_{\infty}$};
  \path (1,0)node[right]{$x_{-}$}--(4,0)node[left]{$x_{+}$};
  \path[every node/.style={draw,circle,inner sep=1pt,fill}] (1,0)node{}--(4,0)node{};
\end{tikzpicture}
  \caption{The limits $u_\infty^{\pm}$ and the removable singularities $x_\pm$.}
  \label{fig:3}
\end{figure}
\begin{proof}
  It is a direct consequence of Lemma \ref{lemma:unif-est} that the $k$th derivative of $u^{-}_{n}$ is bounded on compact subsets of $[0,\infty)$ as $n\to\infty$. Since $k$ is arbitrary, the Arzel\`a-Ascoli theorem implies a subsequence of $u^{-}_{n}$ converges on compact sets in the $C^{\infty}$ topology to a limiting map $u^{-}_{\infty}:[0,\infty)\times \R/\Z\to \cl{B}(1)$.

  Since $u_{n}$ satisfies $\bd_{s}u_{n}+J_{0}\bd_{t}u_{n}=\epsilon_{n}V_{n}$, and $V_{n}\epsilon_{n}\to 0$, we conclude that $u_{\infty}^{-}$ is holomorphic. The infinite cylinder $[0,\infty)\times \R/\Z$ is biholomorphically equivalent to a punctured disk. A standard result in complex analysis says that any bounded holomorphic function defined on a punctured disk has a removable singularity at the puncture.\footnote{In short, the boundedness assumption implies that the Cauchy Integral Formula can be applied to the holomorphic function, without any correction (``residue'') necessary.} Thus $\lim_{s\to \infty}u_{\infty}^{-}(s,t)=x_{-}$ exists (uniformly in $t$). The same argument works for $u^{+}_{n}$. This completes the proof.
\end{proof}


\begin{lemma}\label{lemma:cutting}
  Assume the setup and conclusion of Lemma \ref{lemma:endplate-convergence}. There exists a sequence $\rho_{n}\to +\infty$ so that the restrictions
  \begin{equation*}
    u_{n}^{-}:[0,\rho_{n}]\times \R/\Z\to \cl{B}(1)\text{ and }u_{n}^{+}:[-\rho_{n},0]\times \R/\Z\to \cl{B}(1)
  \end{equation*}
  converge uniformly to the limits $u_{\infty}^{-}$ and $u_{\infty}^{+}$, in the sense that
  \begin{equation*}
    \lim_{n\to\infty}\sup_{(s,t)\in [0,\rho_{n}]\times \R/\Z} \abs{u_{n}^{-}(s,t)-u_{\infty}^{-}(s,t)}+\sup_{(s,t)\in [-\rho_{n},0]\times \R/\Z}\abs{u_{n}^{+}(s,t)-u_{\infty}^{+}(s,t)}=0.
  \end{equation*}
  Moreover, this choice can be arranged so $\epsilon_{n}\rho_{n}\to 0$ and $\rho_{n}<r_{n}$. See Figure \ref{fig:cutting-domain}.
\end{lemma}
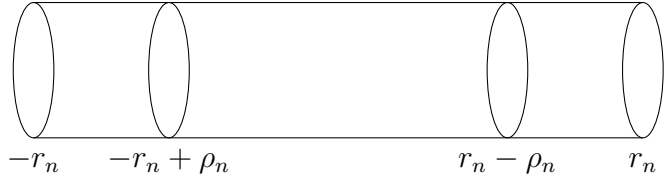
\begin{figure}[H]
  \centering
  \begin{tikzpicture}[scale=.9]
    \draw (0,0) circle (0.3 and 1) (5,0) circle (0.3 and 1) (7,0) circle (0.3 and 1) (-2,0) circle (0.3 and 1);
    \draw (-2,-1)node[below]{$-r_{n}\vphantom{-\rho_{n}}$}--(7,-1)node[below]{$r_{n}\vphantom{-\rho_{n}}$} (-2,1)--(7,1);
    \path (0,-1)node[below]{$-r_{n}+\rho_{n}$}--(5,-1)node[below]{$r_{n}-\rho_{n}$};
  \end{tikzpicture}
  \caption{Cutting the domain into regions as in Lemma \ref{lemma:cutting}. The region in between $-r_n$ and $-r_n+\rho_n$ converges uniformly to $u^-_\infty$ (after translating to the right by $r_n$). The region in between $r_n-\rho_n$ and $r_n$ converges uniformly to $u^+_\infty$ (after translating to the left by $r_n$). In particular, the circles at $\pm (r_n-\rho_n)$ converge uniformly to the removable singularities $x_\pm$.}
  \label{fig:cutting-domain}
\end{figure}

\begin{proof}
  Our definition of $\rho_{n}$ will be such that $\rho_{n}\in \N$. To simplify the notation introduce the quantity:
  \begin{equation*}
    d_{n}(k)=\sup_{(s,t)\in [0,k]\times \R/\Z} \abs{u_{n}^{-}(s,t)-u_{\infty}^{-}(s,t)}+\sup_{(s,t)\in [-k,0]\times \R/\Z}\abs{u_{n}^{+}(s,t)-u_{\infty}^{+}(s,t)}.
  \end{equation*}  
  Consider the following three ($n$-dependent) conditions we can place on $k\in \N$:
  \begin{enumerate}[label=(\arabic*)]
  \item $\epsilon_{n}k\le \epsilon_{n}^{1/2}$,
  \item $d_{n}(k)\le k^{-1}$,
  \item $k<r_{n}$,
  \end{enumerate}
  Define $\rho_{n}$ to be the maximal integer $k$ which satisfies all three properties. We will now show that $\rho_{n}$ has the desired properties. Observe that if $k$ is a fixed number, then both properties will hold for $k$ as $n\to\infty$, since both maps $u_{n}^{\pm}$ converge on compact sets ($[0,k]\times \R/\Z$ are $[-k,0]\times\R/\Z$ are compact sets, and $k^{-1}$ is a fixed number). This shows that $\rho_{n}$ is eventually greater than every fixed number, and hence $\rho_{n}\to\infty$.

  Clearly $\epsilon_{n}\rho_{n}$ tends to zero, since property (1) holds.

  Moreover, since $\rho_{n}\to\infty$, property (2) implies that $d_{n}(\rho_{n})\to 0$. This completes the proof.
\end{proof}

We will now analyze the region $[-r_{n}+\rho_{n},r_{n}-\rho_{n}]\times \R/\Z$. This is where we will see the flow line appear.

\begin{lemma}\label{lemma:final-conv}
  Let $u_{n}$ be a sequence satisfying Definition \ref{definition:set-up} and let $q_{n}$ be the center of mass as in Definition \ref{definition:center-of-mass}. Suppose that $\lim_{n\to\infty}\epsilon_{n}r_{n}=\ell$, and if $\ell=0$, require that $r_{n}\to+\infty$. Fix a choice of subsequence and numbers $\rho_{n}$ so that Lemma \ref{lemma:endplate-convergence} and Lemma \ref{lemma:cutting} hold. After passing to another subsequence, if necessary, the rescaled maps
  \begin{equation*}
    v_{n}(s,t)=u_{n}(\epsilon_{n}^{-1}s,\epsilon_{n}^{-1}t)
  \end{equation*}
  defined on the rescaled cylinder $[\epsilon_{n}(-r_{n}+\rho_{n}),\epsilon_{n}(r_{n}-\rho_{n})]\times \R/\epsilon_{n}\Z$ converge uniformly to a flow line $v_{\infty}$ for $V=\lim_{n\to\infty}V_{n}$ of length $2\ell$. The flow line starts at $x_{-}$ and ends at $x_{+}$.
\end{lemma}
\begin{proof}
  We begin with the a priori estimate for the difference $u_{n}-q_{n}$; recall that we have proved in Lemma \ref{lemma:apriori-exp-fin} that for $n$ sufficiently large and all $k$ we have
  \begin{equation}\label{eq:apriori-final-k}
    \abs{\nabla^{k}(u_{n}-q_{n})(s,t)}\le M_{k}(e^{-c(r_{n}+s)}+e^{-c(r_{n}-s)}),
  \end{equation}
  for $s\in[-r_{n}+1,r_{n}-1]$ and all $t$. In particular, since $\rho_{n}\to \infty$, we conclude 
  \begin{equation*}
    \lim_{n\to\infty} \sup_{(s,t)\in [-r_{n}+\rho_{n},r_{n}-\rho_{n}]\times\R/\Z}\abs{u_{n}(s,t)-q_{n}(s)}=0.
  \end{equation*}
  As this is a $C^{0}$ estimate, it remains true after rescaling, i.e.,\ with $p_{n}(s)=q_{n}(\epsilon_{n}^{-1}s)$ then
  \begin{equation*}
    \lim_{n\to\infty} \sup_{(s,t)\in [-\epsilon_{n}^{-1}(r_{n}-\rho_{n}),\epsilon_{n}(r_{n}-\rho_{n})]\times\R/\epsilon_{n}\Z}\abs{v_{n}(s,t)-p_{n}(s)}=0.
  \end{equation*}  
  Thus to prove the lemma it suffices to show that $p_{n}$ converges uniformly on its domain  $[-\epsilon_{n}^{-1}(r_{n}-\rho_{n}),\epsilon_{n}(r_{n}-\rho_{n})]$ to a flow line for $V$ of length $2\ell$ joining $x_{-}$ to~$x_{+}$.

  Recall that $q_{n}(s)$ satisfies the differential equation \eqref{eq:perturb-flow}, and hence $p_{n}(s)$ satisfies
  \begin{equation*}
    p_{n}^{\prime}(s)-V_{n}(p_{n}(s))=\int_{\R/\Z}V_{n}(u_{n}(\epsilon_{n}^{-1}s,t))-V_{n}(q_{n}(\epsilon_{n}^{-1}s))\,\d t.
  \end{equation*}
  A simple computation \`a la fundamental theorem of calculus yields
  \begin{equation*}
    V_{n}(u_{n}(\epsilon_{n}^{-1}s,t))-V_{n}(q_{n}(\epsilon_{n}^{-1}s))=B_{n}\cdot (u_{n}(\epsilon_{n}^{-1}s,t)-q_{n}(\epsilon_{n}^{-1}s)),
  \end{equation*}
  where $B_{n}$ is uniformly bounded in $n$. Applying \eqref{eq:apriori-final-k} with $k=0$ yields
  \begin{equation*}
    \abs{p_{n}^{\prime}(s)-V_{n}(p_{n}(s))}\le Ce^{-c\rho_{n}}\to 0
  \end{equation*}
  for some constant $C$ for all $s\in [-\epsilon_{n}(r_{n}-\rho_{n}),\epsilon_{n}(r_{n}-\rho_{n})]$. Notice that $p_{n}^{\prime}(s)$ is bounded, and $\lim_{n\to\infty}\epsilon_{n}(r_{n}-\rho_{n})=\ell$. Then a straightforward application of Arz\'ela-Ascoli shows that, after passing to a subsequence, $p_{n}$ converges uniformly to a flow line $v_{\infty}:[-\ell,\ell]\to \cl{B}(1)$ for $V$. To be a bit more precise about the convergence, we have
  \begin{equation*}
    \lim_{n\to\infty}\sup_{s\in [-\epsilon_{n}(r_{n}-\rho_{n}),\epsilon_{n}(r_{n}-\rho_{n})]\cap [-\ell,\ell]}\abs{p_{n}(s)-v_{\infty}(s)}=0.
  \end{equation*}
  In particular, since $p_{n}$ has a bounded derivative, and $\epsilon_{n}(r_{n}-\rho_{n})\to \ell$, we conclude  $p_{n}(\pm\epsilon_{n}(r_{n}-\rho_{n}))$ converges to $v_{\infty}(\pm \ell)$. We have
  \begin{equation*}
    p_{n}(\pm\epsilon(r_{n}-\rho_{n}))=q_{n}(\pm(r_{n}-\rho_{n}))=\int u_{n}(\pm(r_{n}-\rho_{n}),t)\to x_{\pm},
  \end{equation*}
  and hence $v_{\infty}(\pm\ell)=x_{\pm}$, as desired. This completes the proof.  
\end{proof}

\bibliography{citations}
\bibliographystyle{amsalpha}

\end{document}